\documentclass[12pt]{amsart}
\usepackage{graphics}
\usepackage{amsfonts,amssymb,color}
\usepackage[mathscr]{eucal}
\usepackage{amsmath, amsthm}
\usepackage{mathrsfs}
\usepackage{amsbsy}
\usepackage{dsfont}
\usepackage{bbm}
\usepackage{wasysym}
\usepackage{stmaryrd}
\usepackage{url}

\input xypic
\xyoption{all}

\makeindex
\makeglossary

\begin{document}
\baselineskip = 16pt

\newcommand \ZZ {{\mathbb Z}}
\newcommand \NN {{\mathbb N}}
\newcommand \RR {{\mathbb R}}
\newcommand \CC {{\mathbb C}}
\newcommand \PR {{\mathbb P}}
\newcommand \AF {{\mathbb A}}
\newcommand \GG {{\mathbb G}}
\newcommand \QQ {{\mathbb Q}}
\newcommand \bcA {{\mathscr A}}
\newcommand \bcC {{\mathscr C}}
\newcommand \bcD {{\mathscr D}}
\newcommand \bcF {{\mathscr F}}
\newcommand \bcG {{\mathscr G}}
\newcommand \bcH {{\mathscr H}}
\newcommand \bcM {{\mathscr M}}
\newcommand \bcJ {{\mathscr J}}
\newcommand \bcL {{\mathscr L}}
\newcommand \bcO {{\mathscr O}}
\newcommand \bcP {{\mathscr P}}
\newcommand \bcQ {{\mathscr Q}}
\newcommand \bcR {{\mathscr R}}
\newcommand \bcS {{\mathscr S}}
\newcommand \bcV {{\mathscr V}}
\newcommand \bcW {{\mathscr W}}
\newcommand \bcX {{\mathscr X}}
\newcommand \bcY {{\mathscr Y}}
\newcommand \bcZ {{\mathscr Z}}
\newcommand \goa {{\mathfrak a}}
\newcommand \gob {{\mathfrak b}}
\newcommand \goc {{\mathfrak c}}
\newcommand \gom {{\mathfrak m}}
\newcommand \gon {{\mathfrak n}}
\newcommand \gop {{\mathfrak p}}
\newcommand \goq {{\mathfrak q}}
\newcommand \goQ {{\mathfrak Q}}
\newcommand \goP {{\mathfrak P}}
\newcommand \goM {{\mathfrak M}}
\newcommand \goN {{\mathfrak N}}
\newcommand \uno {{\mathbbm 1}}
\newcommand \Le {{\mathbbm L}}
\newcommand \Spec {{\rm {Spec}}}
\newcommand \Gr {{\rm {Gr}}}
\newcommand \Pic {{\rm {Pic}}}
\newcommand \Jac {{{J}}}
\newcommand \Alb {{\rm {Alb}}}
\newcommand \Corr {{Corr}}
\newcommand \Chow {{\mathscr C}}
\newcommand \Sym {{\rm {Sym}}}
\newcommand \Prym {{\rm {Prym}}}
\newcommand \cha {{\rm {char}}}
\newcommand \eff {{\rm {eff}}}
\newcommand \tr {{\rm {tr}}}
\newcommand \Tr {{\rm {Tr}}}
\newcommand \pr {{\rm {pr}}}
\newcommand \ev {{\it {ev}}}
\newcommand \cl {{\rm {cl}}}
\newcommand \interior {{\rm {Int}}}
\newcommand \sep {{\rm {sep}}}
\newcommand \td {{\rm {tdeg}}}
\newcommand \alg {{\rm {alg}}}
\newcommand \im {{\rm im}}
\newcommand \gr {{\rm {gr}}}
\newcommand \op {{\rm op}}
\newcommand \Hom {{\rm Hom}}
\newcommand \Hilb {{\rm Hilb}}
\newcommand \Sch {{\mathscr S\! }{\it ch}}
\newcommand \cHilb {{\mathscr H\! }{\it ilb}}
\newcommand \cHom {{\mathscr H\! }{\it om}}
\newcommand \colim {{{\rm colim}\, }} % colimit
\newcommand \End {{\rm {End}}}
\newcommand \coker {{\rm {coker}}}
\newcommand \id {{\rm {id}}}
\newcommand \van {{\rm {van}}}
\newcommand \spc {{\rm {sp}}}
\newcommand \Ob {{\rm Ob}}
\newcommand \Aut {{\rm Aut}}
\newcommand \cor {{\rm {cor}}}
\newcommand \Cor {{\it {Corr}}}
\newcommand \res {{\rm {res}}}
\newcommand \red {{\rm{red}}}
\newcommand \Gal {{\rm {Gal}}}
\newcommand \PGL {{\rm {PGL}}}
\newcommand \Bl {{\rm {Bl}}}
\newcommand \Sing {{\rm {Sing}}}
\newcommand \spn {{\rm {span}}}
\newcommand \Nm {{\rm {Nm}}}
\newcommand \inv {{\rm {inv}}}
\newcommand \codim {{\rm {codim}}}
\newcommand \Div{{\rm{Div}}}
\newcommand \sg {{\Sigma }}
\newcommand \DM {{\sf DM}}
\newcommand \Gm {{{\mathbb G}_{\rm m}}}
\newcommand \tame {\rm {tame }}
\newcommand \znak {{\natural }}
\newcommand \lra {\longrightarrow}
\newcommand \hra {\hookrightarrow}
\newcommand \rra {\rightrightarrows}
\newcommand \ord {{\rm {ord}}}
\newcommand \Rat {{\mathscr Rat}}
\newcommand \rd {{\rm {red}}}
\newcommand \bSpec {{\bf {Spec}}}
\newcommand \Proj {{\rm {Proj}}}
\newcommand \pdiv {{\rm {div}}}
\newcommand \CH {{\it {CH}}}
\newcommand \wt {\widetilde }
\newcommand \ac {\acute }
\newcommand \ch {\check }
\newcommand \ol {\overline }
\newcommand \Th {\Theta}
\newcommand \cAb {{\mathscr A\! }{\it b}}

\newenvironment{pf}{\par\noindent{\em Proof}.}{\hfill\framebox(6,6)
\par\medskip}

\newtheorem{theorem}[subsection]{Theorem}
\newtheorem{conjecture}[subsection]{Conjecture}
\newtheorem{proposition}[subsection]{Proposition}
\newtheorem{lemma}[subsection]{Lemma}
\newtheorem{remark}[subsection]{Remark}
\newtheorem{remarks}[subsection]{Remarks}
\newtheorem{definition}[subsection]{Definition}
\newtheorem{corollary}[subsection]{Corollary}
\newtheorem{example}[subsection]{Example}
\newtheorem{examples}[subsection]{examples}
\title{Geometry of quintics in $\PR^3$ and the Craighero-Gattazzo surface of general type}
\author{Kalyan Banerjee}

\address{HRI, India}

\email{banerjeekalyan@hri.res.in}

\begin{abstract}
In this paper we study the question whether the tri-canonical system on the Craighero-Gattazzo surface is base point free and at which points does it separate tangent vectors. Also we study the non-rationality of the normalization of the quotient of a general curve (under a given involution) in the product linear system of cubics and quadrics on a singular quintic in $\PR^3$ with four elliptic singularities.
\end{abstract}

\maketitle

\section{Introduction}
The smooth projective surfaces of general type with geometric genus equal to zero are very important from the perspective of algebraic cycles. One of the central open problem in algebraic cycles, the Bloch conjecture predicts that the Chow group of zero cycles on such surfaces is isomorphic to the group of integers. Also another open problem is to classify all surfaces of general type with geometric genus zero.

The recent paper concerns about an example of a surface of general type with geometric genus zero first invented by Craighero and Gattazzo in \cite{CG}. It has also been studied in \cite{DW}. To construct this surface we start with a quintic surface $S$ in $\PR^3$ which is invariant under an involution and having four elliptic singular points. Then we blow up these singularities to obtain a minimal resolution of singularities $V$ equipped with an involution. This $V$ is a surface of general type with geometric genus zero. The aim of this paper is to study the tricanonical system $3K_V$ on the Craighero-Gattazzo surface of general type $V$ with geometric genus zero. To understand whether $3K_V$ is very ample or not. Here we do the calculation about $3K_V$ to deduce that it is base point free and it may not separate tangent vectors at certain points, hence it is not very ample, but we conclude that it is base point free.

Another important thing studied in the paper is the non-rationality of the normalization of the quotient of a general member in the product linear system of cubics and quadrics on  the quintic that we start with. This observation might be helpful to understand the Bloch conjecture on these surfaces.

So the main result of this paper is the following:

\textit{Let $S$ be a quintic in $\PR^3$ invariant under an involution $\sigma$ on $\PR^3$ with four elliptic singularities at the points $[1:0:0:0],[0:1:0:0],[0:0:1:0],[0:0:0:1]$. Let $C$ be a general member of the product linear system of quadrics and cubics on  the quintic $S$ as mentioned in \cite{CG}. Then the normalization of $C/\sigma$ is non-rational.}

This main result is useful in proving Bloch's conjecture for the Craighero-Gattazzo surface which will be dealt in a separate article.

{\small \textbf{Acknowledgements:} The author thanks Vladimir Guletskii for many useful conversations related to the theme of the paper. The author also thanks Indian Statistical Institute Bangalore for hosting this project under the ISF-UGC grant.}

Notation: We work over an algebraically closed ground field of characteristic zero.

\section{Brief recall of the Craighero-Gattazo surface of general type}
The Craighero-Gattazzo surface of general type was discovered by P.Craighero and R. Gattazzo in the paper \cite{CG}. This is a surface of general type with geometric genus and irregularity zero. In the paper by \cite{DW} there is a nice description of this surface starting from a quintic surface in $\PR^3$.

Let $\sigma$ be the automorphism of $\mathbb P^3$ given by the formula,
$$\sigma(X,Y,Z,T)=(T,X,Y,Z)\;.$$
Please note that it is of order $4$, therefore $\sigma^2$ is an involution. Let $S$ be a $\sigma$ invariant quintic in $\mathbb P^3$ such that it has simple elliptic singularities of degree $1$ locally isomorphic to $z^2+x^3+y^6=0$ at the reference points
$$(1,0,0,0),\quad (0,1,0,0),\quad (0,0,1,0),\quad (0,0,0,1)\;.$$
Such a quintic can be given by an explicit equation as mentioned in the paper \cite{DW} by Dolgachev and Werner . Let $\pi:V\lra S$ be the minimal resolution of singularities of this surface $S$. Then following the arguments in the paper by Dolgachev and Werner in \cite{DW}, $V$ is a surface of general type falling under the numerical Godeaux class of such surfaces. It has zero geometric genus and zero irregularity.\\

The fixed point set of the involution $\sigma^2$ is given by the union of lines,
$$r:=\{X+Z=Y+T=0\},\quad r':=\{X-Z=Y-T=0\}\;.$$
The line $r'$ intersects $S$ at five points. Let
$$\{Q_1,Q_2,Q_3,Q_4,Q_5\}\;,$$
be the set of  five points on $V$ corresponding to the five points in $S\cap r'$.
Then by blowing up $V$ at this points we get a new surface $V'$ with a regular map $b:V'\lra V$. The involution $\sigma^2$ acts on $V'$. Let us consider the surface $V'/\sigma^2$. Let
$$p:V'\lra V'/\sigma^2\;,$$
be the canonical projection map. Let us denote $V'/\sigma^2$ by $F$.  It is proved in the paper by Dolgachev and Werner that $F$ is a non-singular rational surface [Proposition 3.1 in \cite{DW}]. Furthermore the proposition $3.1$ in \cite{DW} tell us that there is a linear system of genus three curves on $V'$ which is the pre-image of a linear system of elliptic curves on $F$.

\section{Computation about $3K_V$}
We try to understand the tricanonical system $3K_V$ on $V$ is very ample or not. For that first we try to understand whether it is base-point free or not. We prove:

\begin{theorem}
The tricanonical system $3K_V$ on $V$ is base-point free.
\end{theorem}

\begin{proof}
This follows from \cite{M}[Theorem 2], since the author has come to know this fact after the proof of the theorem in an explicit way, the alternative proof is included.
Let $C_0,C_1,C_2,C_3$ denote the cubics in $\PR^3$ generating a linear system in $\PR^3$ whose pullback minus four elliptic curves give $3K_V$. In precise terms we have
$$C_0=T\{(3r-2)[X+mY+r^2Z]T+(r+1)XY]+(-6r^2+2r+2)XZ+(-2r^2-5r+5)YZ\}$$
$$C_1=X\{(3r-2)[Y+mZ+r^2T]X+(r+1)YZ]+(-6r^2+2r+2)YT+(-2r^2-5r+5)ZT\}$$
$$C_2=Y\{(3r-2)[Z+mT+r^2X]Y+(r+1)ZT]+(-6r^2+2r+2)ZX+(-2r^2-5r+5)TX\}$$
$$C_3=Z\{(3r-2)[T+mX+r^2Y]Z+(r+1)TX]+(-6r^2+2r+2)TY+(-2r^2-5r+5)XY\}$$
Call the four quadrics appearing in the above expression as $Q_0,Q_1,Q_2,Q_3$ then the intersection
$$C_0\cap C_1\cap C_2\cap C_3=(T\cap X\cap Y\cap Z)\cup (T\cap X\cap Y\cap Q_3)$$
$$\cup(T\cap X\cap Z\cap Q_2)\cup(T\cap X\cap Q_2\cap Q_3)$$
$$\cup (T\cap Y\cap Z\cap Q_1)\cup (T\cap Y\cap Q_1\cap Q_3)$$
$$\cup(T\cap Z\cap Q_1\cap Q_2)\cup(T\cap Q_1\cap Q_2\cap Q_3)$$
$$\cup (Q_0\cap X\cap Y\cap Z)\cup(Q_0\cap X\cap Y\cap Q_3)$$
$$\cup (Q_0\cap X\cap Q_2\cap Z)\cup(Q_0\cap Q_3\cap Q_2\cap X)$$
$$(Q_0\cap Q_1\cap Y\cap Z)\cup(Q_0\cap Q_1\cap Q_3\cap Y)$$
$$\cup(Q_0\cap Q_1\cap Q_2\cap Z)\cup(Q_0\cap Q_1\cap Q_2\cap Q_3)$$
Definitely $T\cap X\cap Y\cap Z$ is empty. Now consider the intersection $T\cap X\cap Y\cap Q_3$, that is we put $T=X=Y=0$ in $Q_3$ and it satisfies the equation of $Q_3$ for all $Z$. Therefore we get that this intersection is $[0:0:1:0]$. Similarly we get $[1:0:0:0],[0:1:0:0],[0:0:0:1]$ as the point of intersection of three of the hyperplanes $X=0,Y=0,Z=0,T=0$ with one of $Q_0,Q_1,Q_2,Q_3$.

Now we consider of the type $T\cap X\cap Q_2\cap Q_3$. So putting $T=X=0$ in the equation of $Q_2,Q_3$ we get that
$$YZ=0$$
So either $Y=0$ or $Z=0$ and the intersection contains the points $[0:1:0:0]$ and $[0:0:1:0]$. Similarly intersecting two of $X=0,Y=0,Z=0,T=0$ with two of $Q_0,Q_1,Q_2,Q_3$ we get the points $[1:0:0:0],[0:1:0:0],[0:0:1:0],[0:0:0:1]$ as points of intersection.

Now we consider the intersection of the form $T\cap Q_1\cap Q_2\cap Q_3$. Putting $T=0$ in the equation we get that
$$(XY+mZX+(r+1)YZ)=0$$
$$(3r-2)(ZY+r^2XY)+(-6r^2+2r+2)XZ=0$$
$$(3r-2)(mXZ+r^2YZ)+(-2r^2-5r+5)XY=0$$
which can be written in the matrix form
$$\left(
\begin{array}{ccc}
1 & (r+1) & m \\
r^2(3r-2) & 3r-2 & (-6r^2+2r+2) \\
(-2r^2-5r+5) & r^2(3r-2) & (3r-2)m \\
\end{array}
\right)
$$
whose determinant is
$$(3r-2)[(3r-2)m-r^2(-6r^2+2r+2)]-(r+1)
[mr^2(3r-2)^2-(-6r^2+2r+2)(-2r^2-5r+5)]$$
$$+m(3r-2)[r^4(3r-2)-(-2r^2-5r+5)]$$
this can be written as
$$(3r-2)^2m[1+r^4-r^3-r^2]-(3r-2)(-6r^4+2r^3+2r)$$
$$+
(r+1)(12r^4+26r^3-44r^2+10)-m(3r-2)(-2r^2-5r+5)$$
which is equal to
$$m(r^4(3r-2)^2+6r^3+11r^2-25r+10)+(18r^5-18r^4+4r^3-6r^2+4r)$$
$$+(12r^5+26r^4-44r^3+10r+12r^4+30r^3-44r^2+10)$$
which is
$$m(r^4(3r-2)^2+6r^3+11r^2-25r+10)+(-20r^2+4r+10)$$
now the part
$$m(r^4(3r-2)^2+6r^3+11r^2-25r+10)$$
is zero. Expanding it we get that
$$m(9r^6-12r^5+4r^4+6r^3+11r^2-25r-10)$$
which is
$$m(9r^3-21r^2+25r-19+30r^2-25r+10)=m(9r^3+9r^2-9)=0\;.$$
This is done by using the relation $r^3+r^2=1$. Since the determinant is $-20r^2+4r+10$ and it is relatively prime to $r^3+r^2-1$, it is non-zero. So the intersection of $T=0$ with $Q_1\cap Q_2\cap Q_3$ is $[0:1:0:0]$ and $[0:0:1:0]$ and $[1:0:0:0]$. Similarly we can check that the other intersections like $X\cap Q_0\cap Q_2\cap Q_3$ are one the reference points.

So we are left with only the intersection $Q_0\cap Q_1\cap Q_2\cap Q_3$. For that we prove that the four quadrics are in general position. That is in the parameter space of all quadrics they are linearly independent. So let us have
$$\sum_i c_i Q_i=0$$
which yields the matrix equation in the following form.
$$
\left(
  \begin{array}{cccc}
    a & b & c & d \\
    d & a & b & c \\
    c & d & a & b \\
    b & c & d & a \\
  \end{array}
\right)
$$
here
$$a=(r+1)(3r-2),\quad b=3r-2,\quad c=r^2(3r-2),\quad d=-2r^2-5r+5\;.$$
Since the above matrix is the matrix of a linear operator $T$ on a four dimensional vector space such that $T^4=Id$, we have that the matrix of the operator is non-singular, because it has four distinct eigenvalues. Therefore we conclude that the linear system given by the four cubics in $\PR^3$ has four reference points as four base points. Now when we resolve the singularities we get four rational curves corresponding to four base points. Suppose $3K_V$ has base points then it must come from base points from the linear system of cubics on the quintic $S$, which are precisely four reference points. Since they are blown up to four rational curves and codimension of the base locus of $3K_V$ must be $2$ we have that $3K_V$ is base point free.
\end{proof}

Now we try to understand whether the linear system $L$ given by the four cubics separates tangent vectors on $S$. So we prove the following:

\begin{theorem}
The linear system of cubics on $S$ separates tangent vectors except for the tangent vectors at four reference points.
\end{theorem}

\begin{proof}
For the above it suffices to prove that the linear system separates tangent vectors on $\PR^3$ except for the four reference points. This precisely means that given any closed point  $P$ on $\PR^3$, and given any tangent vector $t$ in $T_P (\PR^3)$, there exists a member $D$ in the given linear system such that $P\in D$ and $t$ does not belong to $T_P (D)$. First observe that for any point $P$ in $\PR^3$ the cubics passing through $P$, in the linear system $L$ forms a sub-linear system $L_P$ of dimension $2$. Now suppose that $C$ is in $L_P$. Then $C$ can be written as
$$\lambda C_0+\mu C_1+\nu C_2+\rho C_3$$
where $C_0,C_1,C_2,C_3$ are the cubics mentioned in \cite{CG} spanning the linear system of cubics on the quintic.
Then
$$d_P(C)=\lambda d_P(C_0)+\mu d_P(C_1)+\nu d_P(C_2)+\rho d_P(C_3)$$
then the collections of $d_P(C)$ with $C$ from $L_P$ determine another linear system $dL_P$. Suppose that we can prove that this linear system is base point free, then it means that there exists $D$  such that $v$ is not in $T_P(D)$ and $P$ belongs to $D$. So it will follow that the linear system $L$ separates tangent vectors.

So suppose that $P$ is a point in $\PR^3$ written as $[x:y:z:t]$ such that $t$ is not zero and it is not $[0:0:0:1]$. Then in terms affine co-ordinates we can write $P$ to $(x,y,z)$. Now we compute the tangent space $T_P (C_0)$ and $T_P(C_1)$.
We have
$$T_P(C_i)=\frac{\partial{C_i}}{{\partial X}}|_{(x,y,z)}(X-x)
+\frac{\partial{C_i}}{{\partial Y}}|_{(x,y,z)}(Y-y)+\frac{\partial{C_i}}{{\partial Z}}|_{(x,y,z)}(Z-z)$$
which is  for $T_P(C_0)$
$$[(3r-2)+(r+1)(3r-2)y+(-6r^2+2r+2)z](X-x)+$$
$$[(3r-2)m+(3r-2)(r+1)x+(-2r^2-5r+5)z](Y-y)+$$
$$[(3r-2)r^2+(-6r^2+2r+2)x+(-2r^2-5r+5)y](Z-z)=0\;.$$
Similarly we have the equation for $T_P(C_1)$ given by
$$[(3r-2)(y+mz+r^2)(1+x)+(r+1)(3r-2)yz+(-6r^2+2r+2)y+(-2r^2-5r+5)z](X-x)$$
$$+[(3r-2)x^2+(3r-2)(r+1)xz+(-6r^2+2r+2)x](Y-y)$$
$$+[(3r-2)mx^2+(r+1)(3r-2)xy+(-2r^2-5r+5)x](Z-z)=0\;.$$
The equation of $T_P(C_2)$ is given by
$$[(3r-2)r^2y^2+(-6r^2+2r+2)yz+(-2r^2-5r+5)y](X-x)+$$
$$[(3r-2)(z+m+r^2x)(1+y)+(3r-2)(r+1)z+(-6r^2+2r+2)zx+(-2r^2-5r+5)x](Y-y)+$$
$$[(3r-2)y^2+(3r-2)(r+1)y+(-6r^2+2r+2)xy](Z-z)=0$$
Now we prove that $T_P(C_0)$ intersects $T_P(C_1)$ in a line. For that we have to prove that there does not exists a polynomial $\lambda(x,y,z)$ such that
$$\lambda(x,y,z)[x(3r-2)(r+1)+(3r-2)m+(-2r^2-5r+5)z]=$$
$$x^2(3r-2)+xz(r+1)(3r-2)+(-6r^2+2r+2)x\;.$$
This we get by writing the equations of $T_P(C_0),T_P(C_1)$ in the matrix form and trying to prove that the two rows we get are linearly independent. So $\lambda=ax+bz+c$ and putting this in the above equation we get that
$$(3r-2)(r+1)a=3r-2$$
$$b(-2r^2-5r+5)=0$$
which gives $a=1/r+1,b=0$
also we have
$$b(3r-2)(r+1)+a(-2r^2-5r+5)=(r+1)(3r-2)$$
which gives us that
$$-2r^2-5r+5=(r^2+2r+1)(3r-2)=3r^3-2r^2+6r^2-4r+3r-2$$
$$=3r^3+4r^2-r-2=r^2-r+1$$
which gives us that
$$3r^2+4r-4=0$$
which is not true. So we get that $T_P(C_0)$ intersects $T_P(C_1)$ in a line. Similarly $T_P(C_0)$ intersects $T_P(C_2)$ in a line and we prove that $T_P(C_1)$ intersects $T_P(C_2)$ in a line which is similar to the previous argument. So we get that $T_P(C_0)\cap T_P(C_1)\cap T_P(C_2)$ intersects only at one point by dimension counting. So we have $dL_P$ is base point free for all $P=[x:y:z:t]$ with $t$ not equal to zero and not $[0:0:0:1]$. For $[0:0:0:1]$ writing the equations for $T_P(C_i)$ for $i=1,2,3$ we get that the equations are $X=Y=Z=0$, which again intersect at the origin. So for $t$ not equal to zero we have $dL_P$ base point free. Similarly argument for $P$ with $x,y,z$ not equal to zero gives us that $dL_P$ is base point free. So $L$ separate tangent vectors. So for all points of $S$ which are not reference points the tangent vectors will be separated by $3K_V$.
\end{proof}

\subsection{Computation regarding $5K_V$}

We first determine the bundle $K_V$, interms of the pull-back of a $\sigma^2$ invariant hyperplane section and four elliptic curves. By the localization exact sequence and by using the fact that blow up is a birational map we have that
$$K_V=\pi^*(H)+\sum_{i=1}^4 n_iE_i$$
we find out the $n_i$ by Adjunction formula. We have
$$K_V.E_i=-n_i$$
and
$$E_i^2=-1\;.$$
So putting these in  the adjunction formula for $E_i$ we have
$$-1-n_i=0$$
which leads us to the fact that
$$n_i=1\;.$$
So we have that
$$K_V=\pi^*(H)-\sum_i E_i\;.$$
Now we compute $2K_V$. It is the pullback of the pencil of quadrics $\bcQ$. So we have
$$2K_V=\pi^*(\bcQ)+\sum_i n_iE_i$$
we have that
$$2K_V.E_i=-n_i\;.$$
This gives by adjunction formula that
$$-2-n_i=0$$
that is $n_i=-2$.
Now we do the same for the linear system associated to $3K_V$. We know that it is the pull-back of the $3$ dimensional linear system $\bcC$ of cubics.
That is
$$3K_V=\pi^*(\bcC)+n_iE_i$$
then we have
$$3K_V.E_i=-n_i$$
and we have $n_i=-3$.
So we have that
$$2K_V=\pi^*(\bcQ)-2\sum_i E_i$$
and
$$3K_V=\pi^*(\bcC)-3\sum_i E_i$$
therefore we have that
$$5K_V=\pi^*(\bcC.\bcQ)-5\sum_i E_i\;.$$

\section{Non-rationality of the push-down}
We consider the  linear system of quintics obtained by multiplying the generators of the pencil of quadrics $\bcQ$ and the pencil of cubics $\bcC$ restricted on the quintic. It has eight base points. So the intersection of a quintic from this linear system with the quintic $S$ are singular precisely at these base points. Now we can choose the quintic to be not containing the five fixed point of intersection of the intersection of $r'$ (one component of the fixed locus of $\sigma^2$). The other line $r$ containing the fixed points of $\sigma^2$ is contained in the quintic $S$. So its intersection with a general member of the linear  system is atmost five points. Since the points in the base locus are not contained in the fixed locus of $\sigma^2$ we get that the point in the fixed locus of $\sigma^2$ are non-singular points of the general curves from the linear system. Let us denote the involution $\sigma^2$ by $\tau$, a general member of the linear system on the quintic by $C$. Then we have following commutative diagram.
$$
  \diagram
   \wt{C}\ar[dd]_-{} \ar[rr]^-{} & & \wt{C/\tau} \ar[dd]^-{} \\ \\
  C \ar[rr]^-{} & & C/\tau
  \enddiagram
  $$
\begin{theorem}
We prove that the genus of $\wt{C/\tau}$ is positive and hence it is non-rational.
\end{theorem}

\begin{proof}
The Riemann-Hurwitz formula for non-singular curves give us that
$$2p_g(\wt{C})-2=2(2p_g(\wt{C/\tau})-2)+\deg (R)$$
where $R$ is the ramification divisor. Since the quasi-projective curve obtained by throwing out the singular points and their images under the action of $\tau$ is non-singular. Its image under the quotient map from $C$ to $C/\tau$ is also non-singular. Therefore for any non-singular point $Q$ on $C/\tau$ such that $P$ is mapped to $Q$, and $t$ is a uniformising parameter of the maximal ideal  in $\bcO_Q$, we have
$$e_{\phi}(P)=v_p(\phi^*(t))$$
which is either one or two.
Now degree of the ramification divisor is
$$\sum_{Q\in R}(e_{\phi}(P)-1)\;.$$
Since by the Riemann-Hurwitz formula $\deg(R)$ is even we have only two or four fixed points of the quotient map (the set of singular points of $C$ is disjoint from the set of fixed points on $C$, so the ramification divisor in $\wt{C}$ can have atmost five points) and in that case $R$ consists of either $2$ or $4$ points.
Now we have
$$p_g(\wt{C})=p_a(C)+\sum_{P\in C}\delta_P$$
where $\delta_p$ is the length of $\wt{\bcO_P}/\bcO_P$. Now since $C$ is the complete intersection of two quintics its arithmetic genus is
$$5.5(5+5-4)/2+1=76\;.$$
Now we understand the relation between $\delta_P$ and $\delta_Q$, where $P$ is mapped to $Q$. Consider the map from $\bcO_Q$ to $\bcO_P$. Since we have a $2:1$ map we have that the degree of the field extension $[k(C):k(C/\tau)]=2$. Then we prove that $\wt{\bcO_P}$ is a degree two extension over $\wt{\bcO_Q}$, where these are the integral closures of $\bcO_P$ and $\bcO_Q$ respectively. So suppose that $\alpha$ is a root of a quadratic polynomial over $k(C/\tau)$ which generate $k(C)$. Say
$$\alpha^2+b\alpha+c=0$$
writing $b=b_1/b_2,c=c_1/c_2$ where $b_i,c_i$ belong to $\wt{\bcO_Q}$ we get that $b_2c_2\alpha$ is integral over $\wt{\bcO_Q}$. So it belongs to $\wt{\bcO_P}$. Now since $\wt{\bcO_Q}$ is local ring (since it is integral closure of $\bcO_Q$ in $k(C/\tau)$) we have that $mb$ belong to $\wt{\bcO_Q}[1+a\alpha]$, where $a=b_2c_2$, $m$ belongs to the maximal ideal of $\wt{\bcO_Q}$, $b$ in $\wt{\bcO_P}$. So by Nakayama lemma we have that $\wt{\bcO_P}=\wt{\bcO_Q}[1,a\alpha]$. Since the quotient map is $2:1$ hence finite, similar argument as above shows that we have $\bcO_P=\bcO_Q[1,b\alpha]$. This gives us that $\delta_P=2\delta_Q$. Suppose that $R$ consists of two points. So putting this in the above formula we get that
$$2p_a(C)-2=2(2p_a(C/\tau)-2)+\deg(R)+2\sum_{P\in C} \delta_P-4\sum_{Q\in C/\tau}\delta_Q$$
putting the values we have
$$150=2(2p_g(\wt{C/\tau})-2)-4\sum_i \delta_{Q_i}+2+2\sum_{P\in C}\delta_P+4\sum_{i=1}^4(\delta_{Q_i})$$
Here the eight points of singularities of $C$ are mapped to $Q_1,Q_2,Q_3,Q_4$.
So we get that
$$75+2-1=2p_g(\wt{C/\tau})+(\sum_{P\in C}\delta_{P})$$
suppose that $p_g(\wt{C/\tau})=0$, then we have
$$19=(\sum_{P\in C}\delta_{P})$$
We have to prove that this case does not occur.

For the case when $R$ consists of four points we have by the above
$$75+2-2=75=4(\sum_{P\in C}\delta_{P})$$
which is  absurd because we have two set of singular points $A,B$ of $C$ each consisting $4$ points, which are in the same orbit under the action of $\Sigma$. Then the sum in the right hand side in the above is
$$4(\delta_P+\delta_Q)$$
where $P$ belongs to $A$ and $Q$ belongs to $B$. Since $4$ does not divide $75$, the above inequality cannot be true.

So we only have to prove that for a general member of $\bcD$, its intersection with the line $r$ given by $X=-Z,Y=-T$ consists of four points. First of all we can write the linear system of quintics $\bcD$ as a union of $Z_4\cup Z_2$, where $Z_4,Z_2$ contains the quintics having $4,2$ intersection points with $R$ respectively. We prove  that $Z_2$ is Zariski closed and proper inside $\bcD$ that is $Z_4$ is non-empty.

So consider the pencil generated by $XZ C_0,YT C_1$, that is all linear combinations of the form
$$\lambda XZ C_0+\mu YT C_1$$
then  intersecting such a quintic with $R$, that is putting $X=-Z,Y=-T$
we have
$$XY(\lambda XQ_0-\mu YQ_1)=0$$
where $Q_0,Q_1$ are quadrics appearing in the expression of $C_0,C_1$.
so either $XY=0$, that is either $X=0=Z$ or $Y=0=T$, in which case we obtain two points $[1:0:-1:0],[0:1:0:-1]$ or we have
$$\lambda XQ_0-\mu YQ_1=0$$
writing this equation in terms of polynomials of $\lambda,\mu$ we have that
$$a(\lambda,\mu)X^3+b(\lambda,\mu)X^2+c(\lambda,\mu)X+d(\lambda)=0$$
where $a,b,c,d$ are linear polynomials in $\lambda,\mu$. Now we already have two distinct closed points of intersection of a member of the pencil with the fixed locus. Therefore to have exactly two intersection points with the fixed locus we have the above cubic equation has one root. That can be prescribed by the condition $$9d(\lambda,\mu)a(\lambda,\mu)-b(\lambda,\mu)c(\lambda,\mu)=0$$
Since this is a quadratic in $\lambda,\mu$, so it has two solutions. Hence the set of quintics in the pencil having four intersection points is non-empty, hence $Z_4$ is non-empty. So now we prove that $Z_2$ is  Zariski closed. That is we consider all $[\lambda_0:\lambda_1:\cdots:\lambda_7]$ in $\PR^7$ such that the zero locus of
$$\sum_i \lambda_i \bcQ_i$$
intersects $R$ at two points, where $\bcQ_i$'s are the quintics generating the product linear system. So setting $X=-Z,Y=-T$ in the above equation we get that
$$\sum_i \lambda_i \bcQ_i(X,Y,-X,-Y)=0$$
So we write it down as
$$f(X,Y)=a_0X^5+a_1X^4Y+a_2X^3Y^2+a_3X^2Y^3+a_4XY^4+a_5Y^5=0$$
where $a_i$'s are linear polynomials in $\lambda_i$'s. Since this quintic does not intersect $X=0,Y=0$, it is inside $U\cap V$, where $U,V$ are affine open in $\PR^1$ defined by $X\neq 0,Y\neq 0$.
Suppose the locus of this quintic  has only two closed points. It means that we have two cases.

1) When the equation defined by putting $Y=1$ in the above quintic has one root and the equation defined by putting $X=1$ has one root, and the two closed points corresponding to these two roots are distinct.
In this case we have say $\alpha$ be a root of $f(X,1)$. In that case we get that
$$5\alpha=-a_1/a_0,\quad 10\alpha^2=a_2/a_0,\quad 10 \alpha^3=-a_3/a_0$$
$$5\alpha^4=a_4/a_0,\quad \alpha^5=a_5/a_0\;.$$
So we get that
$$25\alpha^{10}=\frac{a_1a_4a_5}{a_0^3}$$
and
$$100\alpha^5=\frac{a_2a_3}{a_0^2}\;.$$
Eliminating $\alpha$ we get that
$$a_2^2a_3^2=400a_0a_1a_4a_5\;.$$
Similarly if we have $f(1,Y)$, then by doing the same we have the relation
$$a_2^2a_3^2=400a_0a_1a_4a_5\;.$$
This defines a quartic hypersurface in $\PR^7$.

2) The two closed points corresponding to these two roots are the same. Then we have $f(X,1)=f(1,Y)=0$ and these two polynomials has the same roots. In this case we have two roots of the polynomial $f(X,1)$, say $\alpha$, $1/\alpha$, one of multiplicity two and other of multiplicity three. Then the relations between roots and coefficients are
$$3\alpha+2(1/\alpha)=-a_1/a_0$$
$$\alpha^3/\alpha^2=\alpha=a_5/a_0$$
that would mean eliminating $\alpha$ we have
$$3a_5/a_0+2a_0/a_5=-a_1/a_0$$
that is
$$3a_5^2+2a_0^2=-a_1a_5\;.$$
So it defines a quadric in $\PR^7$. So $Z_2$ is Zariski closed and $Z_4$ is non-empty. Since $\bcD\cong \PR^7$ is the union $Z_4\cup Z_2$, we have that $Z_4$ is a non-empty Zariski open subset of $\PR^7$. So for a general member of the linear system of quintics we have the genus of $\wt{C/\tau}$ is positive.
\end{proof}

%{\small \textbf{Acknowledgements:} On behalf of all authors, the corresponding author states that there is no conflict of interest.}

\end{document}